\documentclass[12pt]{amsart}

\usepackage{amsfonts}
\usepackage{amscd}
\usepackage{amssymb}
\usepackage{amsthm}
\usepackage{amsmath}
\usepackage{comment}
\usepackage{epsfig}
\usepackage[all]{xy}
\usepackage{fourier} % math & rm
\usepackage[scaled=0.875]{helvet} % ss
\newtheorem{Theorem}{Theorem}
\newtheorem{Proposition}[Theorem]{Proposition}
\newtheorem{Lemma}[Theorem]{Lemma}

\newtheorem{Corollary}[Theorem]{Corollary}

\theoremstyle{definition}

\newtheorem{Definition}[Theorem]{Definition}

\theoremstyle{remark}
\newtheorem*{Remark}{Remark}

\makeatletter
\renewcommand{\@makefnmark}{\mbox{\textsuperscript{}}}
\makeatother

\newcommand{\half}{0.5}
\newcommand{\asl}{\widehat{\mathfrak{sl}}}

\newcommand{\g}{\mathfrak{g}}
\newcommand{\n}{\mathfrak{n}}

\newcommand{\ft}{\mathfrak{t}}
\newcommand{\fb}{\mathfrak{b}}
\newcommand{\fp}{\mathfrak{p}}
\newcommand{\B}{\mathfrak{B}}
\newcommand{\F}{\mathfrak{F}}
\newcommand{\oF}[1]{\overset{\circ}{\F}{}^{#1}}
\newcommand{\Gr}{\mathcal{G}\emph{r}}
\newcommand{\cL}{\mathcal{L}}
\newcommand{\Z}{\mathbb{Z}}
\newcommand{\K}{\mathcal{K}}
\newcommand{\cZ}{\mathcal{Z}}
\newcommand{\cN}{\mathcal{N}}
\newcommand{\cV}{\mathcal{V}}
\newcommand{\cU}{\mathcal{U}}
\renewcommand{\O}{\mathcal{O}}
\newcommand{\GK}{G\left(\K\right)}

\newcommand{\C}{\mathbb{C}}

\newcommand{\Irr}{\text{Irr}}

\newcommand{\val}{\text{val}}

\newcommand{\vgamma}{\langle\gamma|}
\newcommand{\vmu}{\langle\mu|}
\newcommand{\Stot}{\mathbb{S}_{\mathfrak{g},\mathfrak{b},{\mathfrak{p}_i}}}
\newcommand{\Fock}{\mathbf{F}}
\newcommand{\wt}{\text{wt}}

\newcommand{\POne}{\mathbb{P}^1}
\newcommand{\AOne}{\mathbb{A}^1}

\begin{document}

\title{Double MV Cycles and the Naito-Sagaki-Saito Crystal}

\author{Dinakar Muthiah}

\begin{abstract}
The theory of MV cycles associated to a complex reductive group $G$ has proven to be a rich source of structures related to representation theory. We investigate double MV cycles, which are analogues of MV cycles in the case of an affine Kac-Moody group. We prove an explicit formula for the Braverman-Finkelberg-Gaitsgory \cite{BFG} crystal structure on double MV cycles, generalizing a finite-dimensional result of Baumann and Gaussent \cite{baumanngaussent}. As an application, we give a geometric construction of the Naito-Sagaki-Saito \cite{NSS} crystal via the action of $\widehat{SL}_n$ on Fermionic Fock space. In particular, this construction gives rise to an isomorphism of crystals between the set of double MV cycles and the Naito-Sagaki-Saito crystal. As a result, we can independently prove that the Naito-Sagaki-Saito crystal is the $B(\infty)$ crystal. In particular, our geometric proof works in the previously unknown case of $\asl_2$.
\end{abstract}

\maketitle

%\tableofcontents

\section{Introduction}

\subsection{Crystals} Let $\g$ be a Kac-Moody Lie algebra. Kashiwara invented the notion of a $\g$-\emph{crystal}, which is a combinatorial analogue of a $\g$ representation. A crystal consists of a set $\mathbf{B}$ along with crystal operators and some auxiliary data (we will review the precise definition in section \ref{ReviewOfCrystals}). A rich source of crystals comes from actual representations:  given an integrable representation in the BGG category $\O$ or a Verma module, we can canonically extract a crystal using Kashiwara's method of \emph{crystal bases}. 

Let us give labels to two crystals that will be of interest to us. Using crystal bases, define $B(\lambda)$ to be the crystal associated to the irreducible representation $V(\lambda)$ of highest weight $\lambda$, and define $B(\infty)$ to be the crystal associated to the  Verma module of highest weight zero (equivalently, the negative part of the universal enveloping algebra). By general theory of crystals, there is a recipe for recovering the various $B(\lambda)$ crystals from the crystal $B(\infty)$, and vice versa. Because of this equivalence, we will focus on the $B(\infty)$ crystal in our discussion.

An interesting problem is to find algebraic, geometric, and combinatorial realizations of the crystal $B(\infty)$. In the next subsection, we will review an algebraic realization of the crystal coming from Lusztig's canonical basis. In the subsequent subsection, we will review a geometric realization of the crystal coming from MV cycles in the affine Grassmannian. In the course of the discussion, we will explain how both realizations give rise to the same combinatorial realization of the $B(\infty)$ crystal via MV polytopes.

\subsection{Lusztig's Canonical Basis}

 One particularly nice realization of the $B(\infty)$ crystal comes from Lusztig's canonical basis \cite{LusCanonical}, which is a basis in the negative part of the quantum group enjoying many nice properties. When $\g$ is finite type, there is a natural parameterization of the canonical basis coming from the various Poincar\'e-Birkhoff-Witt (PBW) bases. An interesting question is to study the combinatorics that records how to pass between the various parameterizations of the canonical basis coming from the various PBW bases. Lusztig gave an explicit answer in simply-laced cases \cite{LusPiecewise}. Berenstein and Zelevinsky \cite{BZ} gave an answer in all types. In addition, they indicate how the reparameterization data can be arranged into a combinatorial gadget called an \emph{MV polytope} (we will explain this name, due to Kamnitzer, in the next subsection). In particular, they produce a bijection between the canonical basis and the set of MV polytopes. Finally, they give a combinatorial description of the crystal structure using only the data coming from MV polytopes.

\subsection{MV Cycles in the Finite Dimensional Case}

Let $G$ be a complex reductive group. The geometric Satake equivalence of Lusztig \cite{Lus}, Beilinson-Drinfeld \cite{BD}, Ginzburg \cite{Ginz}, and Mirkovic-Vilonen \cite{MV} relates the geometry of the affine Grassmannian with the representation theory of the dual group $G^\vee$. The most recent proof, due to Mirkovi\'c-Vilonen, provided even finer information; they gave an explicit basis for each irreducible representation of $G^\vee$ indexed by certain irreducible subvarieties of the affine Grassmannian of $G$. These irreducible subvarieties are called \emph{Mirkovi\'c-Vilonen (MV) cycles}, and they are highly structured.

\begin{Remark} 
MV cycles come in two flavors: there are those that correspond to basis vectors in irreducible representations, and there are those that correspond to basis vectors in Verma modules. In this paper we will focus almost exclusively on MV cycles corresponding to basis vectors in Verma modules. Unless we specify otherwise, we will mean this latter variety when we write ``MV cycles''.
\end{Remark}

\noindent Let us highlight some key results in the theory of MV cycles, which will be relevant to our later discussion:

\begin{itemize}
\item Braverman, Finkelberg, and Gaitsgory \cite{BravGait,BFG} proved that the MV cycles corresponding to basis vectors in irreducible representations carry a natural crystal structure for the dual Lie algebra. More specifically, corresponding to each irreducible representation $V(\lambda)$ of $G^\vee$ with highest weight $\lambda$, they endow the set of MV cycles corresponding to $V(\lambda)$ with the structure of the crystal $B(\lambda)$.

As we discussed earlier, there is a natural way to extract a $B(\infty)$ crystal given a suitable family of crystals $\left\{B(\lambda)\right\}$. The geometric counterpart is exactly the process of passing from MV cycles corresponding to irreducible representations to MV cycles corresponding to Verma modules. 

\item Kamnitzer \cite{Kam}, following initial work of Anderson \cite{And}, studied MV cycles via their moment polytopes, which he calls MV polytopes. In particular, he gave an explicit combinatorial description of all MV polytopes. He discovered that MV polytopes are precisely the polytopes considered by Berenstein and Zelevinsky in their study of the canonical basis (hence the anachronistic name of the previous section). As a result, he obtained a natural bijection between MV cycles and the canonical basis using MV polytopes as an intermediary. Consequentially, the set of MV polytopes acquire two crystal structures: one arises from the Braverman-Finkelberg-Gaitsgory crystal structure on MV cycles, and the other comes from the crystal structure on Lusztig's canonical basis. In \cite{KamCrystal}, Kamnitzer proved that these two crystal structures in fact agree.

\item Baumann and Gaussent in \cite{baumanngaussent} gave explicit formulas for the crystal structure on MV cycles. They gave a recipe to construct a MV cycle from its crystal-theoretic string parameterization. They were able to reproduce many of Kamnitzer's results using their formula, and they compared the crystal structure on MV cycles with that coming from the theory of LS galleries. 

\item Hong in \cite{Hong} investigated how MV cycles and MV polytopes behave when you pass from $G$ to $G^\sigma$, the fixed point group of a Dynkin diagram automorphism. He noted that the Dynkin diagram automorphism also acts on MV cycles and MV polytopes, He proves that MV cycles and MV polytopes for $G^\sigma$ are canonically identified with those of $G$ that are fixed under the diagram automorphism.

\end{itemize}

\subsection{MV Cycles in the Affine Case}

Let us consider a general Kac-Moody group $G$. In this generality, there is currently no way to speak of the affine Grassmannian as a geometric object, and there is no geometric Satake equivalence. However,  because the representation theory of a general Kac-Moody group is formally very similar to that of a finite-dimensional reductive group, it is natural to expect that some of the geometric theory should generalize.

Indeed, we can define a substitute for MV cycles using spaces of maps from the projective line into flag schemes (we will call them \emph{map spaces} from now onward). When $G$ is finite-dimensional reductive group, we have an isomorphism between MV cycles and irreducible components of certain map spaces \cite{FM}. So for many purposes, it suffices to study the map spaces. The main benefit of the map spaces is that they can be defined for any symmetrizable Kac-Moody group. 

Thus the starting point for our discussion is replacing MV cycles with their map space analogues. Unfortunately, unless $G$ is finite dimensional or of untwisted affine type, there is little that we can say about the map space analogues of MV cycles. 

In this untwisted affine case, let us call the map space analogues of MV cycles \emph{double MV cycles}. Braverman-Finkelberg-Gaitsgory \cite{BFG} proved that double MV cycles are finite-dimensional schemes. They proved that all double MV cycles of a given weight  have the same dimension, and they compute that dimension to be exactly the number predicted from the study of ordinary MV cycles. With this computation, they were able to define a crystal structure on the set of double MV cycles in a way completely analogous to their construction when $G$ is finite dimensional. Moreover, they are able to prove that the resulting crystal is the $B(\infty)$ crystal for the dual Lie algebra.

We should mention at this point that Nakajima has a construction of MV cycles in the untwisted affine type A case \cite{Nak} via quiver varieties. His construction is the analogue of MV cycles corresponding to basis vectors in irreducible representations (in contrast to the Braverman-Finkelberg-Gaitsgory construction, which is the analogue of MV cycles corresponding to basis vectors in Verma modules).

\subsection{The Naito-Sagaki-Saito Crystal}

It is natural to ask what are the analogues of MV polytopes in the general Kac-Moody case. In the affine case, there is still the notion of Lusztig's canonical basis. However, the notion of a PBW parameterization is much more complicated because of the appearance of imaginary roots \cite{BCP,BN}, and the corresponding combinatorics is still unsolved.  

In the case of $\asl_n$,  Naito, Sagaki, and Saito \cite{NSS} develop an approach to MV polytopes as follows.  Consider the Dynkin diagram $A_\infty$, which is the type-A Dynkin diagram with nodes extending infinitely in both directions. They develop a reasonable candidate for $A_\infty$ MV polytopes, which they call $A_\infty$-Berenstein-Zelevinsky (BZ) data. They do this by noticing that the Dynkin diagram $A_\infty$ behaves in many ways like a finite type Dynkin diagram. In particular, all its finite subdiagrams are finite-type Dynkin diagrams. Using this observation, they are able to generalize the combinatorial characterization of MV polytopes in the finite-type case to the case of $A_\infty$.

They then observe that the Dynkin diagram for $\asl_n$ is obtained by ``folding'' the Dynkin diagram $A_\infty$ by an automorphism; namely, the automorphism is given by shifting the nodes of $A_\infty$ right by $n$ positions. This automorphism also acts naturally on the set of $A_\infty$-BZ data. Naito-Sagaki-Saito then consider the set of all $A_\infty$-BZ data fixed by this automorphism, and define crystal operators on this set in a natural way. Passing to a subset, they construct a crystal for $\asl_n$, which they prove is isomorphic to the $B(\infty)$-crystal when $n > 2$. They call the resulting crystal the set of ``affine Berenstein-Zelevinsky data''. In this paper, we will refer to their construction as the \emph{Naito-Sagaki-Saito (NSS) crystal}. Their methods are purely combinatorial, relying crucially on a result of Stembridge that characterizes the $B(\infty)$-crystal in the simply-laced affine case. In particular, they do not apply when $n=2$.

It is worth noting that the result of their construction is very similar in flavor to the results of \cite{Hong} that describe how finite-type MV polytopes behave under Dynkin diagram automorphisms.

\subsection{Main Results}

The first result of this paper is a generalization of the crystal operator formula of Baumann-Gaussent \cite{baumanngaussent} to the case of double MV cycles. This formula is valid in all untwisted affine cases. We hope that it will be useful for further investigations of double MV cycles.

The second result of this paper is an application of this formula to the case of type A double MV cycles. Here we show that we can extract the Naito-Sagaki-Saito crystal structure from double MV cycles using the action of the Kac-Moody group $\widehat{SL}_n$ on Fermionic Fock space. This gives rise to an explicit isomorphism of crystals between the set of double MV cycles and the Naito-Sagaki-Saito crystal. Using this isomorphism, we offer an independent proof of the fact that the Naito-Sagaki-Saito crystal structure gives rise to the $B(\infty)$ crystal. Furthermore, our proof includes the case of $n=2$, which is new.

\subsection{Heuristic Motivation} As we mentioned before, Naito-Sagaki-Saito's construction gives a combinatorial generalization of the work of Hong \cite{Hong} describing how MV polytopes behave under diagram automorphism. Our result is a partial geometric generalization of his work. 

Kamnitzer proved \cite{Kam} that MV cycles are determined by computing valuations when acting on extremal weight vectors of fundamental representations of $G$. Let us denote by $\widetilde{GL_\infty}$ the Kac-Moody group corresponding to the Dynkin diagram $A_\infty$. Since Fermionic Fock space is the direct sum of all fundamental representations of $\widetilde{GL_\infty}$, it is natural to expect that computing valuations on this representation should determine MV cycles for this group. Via the embedding $\widehat{GL}_n \subset \widetilde{GL_\infty}$ coming from fixed points of the diagram automorphism, following Hong, we should expect that MV cycles of $\widehat{GL}_n$ should be the fixed points for MV cycles of $\widetilde{GL_\infty}$ (this includes the case of MV cycles for $\widehat{SL}_n$, which should be a subset of those for $\widehat{GL}_n$). In particular, they should be determined by computing valuations on Fermionic Fock space. Unfortunately, we have no good candidate for MV cycles for $\widetilde{GL_\infty}$, so this argument is only heuristic. However, in this paper we do prove that MV cycles for $\widehat{SL}_n$ are determined by computing valuations on Fermionic Fock space, and that these valuations give rise to the Naito-Sagaki-Saito crystal structure. This is exactly what we should expect from Hong's results in the finite-dimensional case.

\subsection{Remarks on Open Problems} An open problem is to understand double MV cycles in all types. In particular, we would like to construct an analogue of MV polytopes for all affine types, i.e. we would like to extract from each double MV cycle a combinatorial gadget such that the crystal structure on double MV cycles corresponds to a combinatorial-defined crystal structure on the combinatorial gadgets.  An even more ambitious goal would be to connect any analogue of MV polytopes to the PBW parameterization of the affine canonical basis explained in \cite{BCP,BN}. 

If we follow Kamnitzer's results in the case of a finite-dimensional group, we should expect that double MV cycles for an affine Kac-Moody group $G$ should be determined by computing valuations when acting on extremal weight vectors of fundamental representations of $G$. However, this is not the case, even for $\widehat{SL}_2$. For $\widehat{SL}_2$, explicit computations indicate that only partial information is obtained by computing valuations on extremal weight vectors. That partial information seems to correspond to the ``real part'' of the PBW parameterization of the affine canonical basis. The missing part has to do with imaginary roots, and it is currently unclear how to recover this part using double MV cycles.

We should remark that the results of this paper in the case of $\widehat{SL}_n$ show that although computing valuations on extremal weight vectors of fundamental representations do not suffice to determine MV cycles, it does suffice to compute valuations on a larger set of vectors in fundamental representations. This is true because the action of  $\widehat{SL}_n$ on Fermionic Fock space decomposes as a direct sum of copies of fundamental representations. An open question is to see if the analogue of this statement is true in all types, i.e. are double MV cycles in all types determined by computing valuations on some set of vectors in fundamental representations? Based on our results in type A, we expect the answer to be yes. 

We hope that the results of this paper can come to bear on this problem in two different ways. First, our explicit formula for the crystal structure on double MV cycles applies in all untwisted affine types. We expect it to be useful when proving results about combinatorial data related to double MV cycles. Indeed, it is precisely the tool we used to prove the connection to the NSS crystal. 

Second, we can try to apply our understanding of the connection between double MV cycles and the NSS crystal to the PBW parameterization of the canonical basis. In principle, we have a bijection between the NSS crystal and the set of PBW parameterization data coming from transporting the crystal structure. Understanding this bijection explicitly should be an algebraic/combinatorial problem. If we can understand this bijection, we can compose it with our know bijection between double MV cycles and the NSS crystal. The result would be an explicit bijection between double MV cycles and PBW parameterization data. As we mentioned before, our explicit computations for $\widehat{SL}_2$ seem to indicate how this bijection will work for the ``real part'' of PBW parameterization data. 

Another open problem is to write double MV cycles of a given weight as a disjoint decomposition of locally closed subsets. Kamnitzer accomplishes this in the case of a finite-dimensional group \cite{Kam} after choosing a reduced decomposition for the longest element of the Weyl group. This choice of reduced decomposition corresponds exactly to the choice needed to construct the PBW basis. In the affine case, an analagous choice is needed to construct the PBW basis \cite{BCP,BN}. We hope that a construction of MV polytopes in the affine case, along with a bijection to the PBW parameterization of the canonical basis, will lead to a generalization of Kamnitzer's disjoint decomposition. Using such a decomposition, one could count finite-field-valued points in double MV cycles. Such a computation could lead to an alternate proof of the affine Gindikin-Karpelevich formula (c.f. \cite{BFK}) and perhaps explain its mysterious form. 

\subsection{Organization of the Paper}
In the second section, we review the definition of MV cycles, and we review the theory of quasimaps spaces. In the third section, we recall the Braverman-Finkelberg-Gaitsgory crystal structure. We prove a direct generalization of Baumann-Gaussents explicit formula for the crystal structure to the affine case.

For the remainder of the paper, we focus on the type-A affine case. In section four, we review the Fermionic Fock space and recall explicit formulas for the action of $\asl_n$ on this vector space. In section five, we rephrase the Naito-Sagaki-Saito crystal structure using the language of Maya diagrams. In section six, we show how this crystal structure arises from the action of $\widehat{SL}_n$ on Fermionic Fock space. We prove that this construction gives rise to an isomorphism of crystals between the set of double MV cycles and the Naito-Sagaki-Saito crystal. Using this isomorphism, we reprove that the Naito-Sagaki-Saito crystal structure is the $B(\infty)$-crystal.

\subsection{Acknowledgments}

I thank Michael Finkelberg, Joel Kamnitzer, and Peter Tingley for many useful conversations leading up to the completion of this work. I would like to particularly thank Joel Kamnitzer for pointing me to the work of Baumann and Gaussent, which proved to be the key tool in completing this work. I thank Peter Tingley for letting me modify his pictures of Maya diagrams and charged partitions from his expository notes on Fock space \cite{Tingley}. 

I would like to especially like to thank my advisor, Alexander Braverman, for his guidance and for numerous useful conversations throughout the course of this work. Without his steady patience and encouragement I would not have completed this work.

Finally, I thank the referee for useful comments and suggestions that have improved this paper.

Much of this work was completed during the semester program ``Langlands Duality in Representation Theory and Gauge Theory'' at the Institute for Advanced Studies at the Hebrew University of Jerusalem in the fall of 2010. I would like to thank the organizers of the program and the staff of IAS for their hospitality.

\section{Preliminary Notions}

\subsection{Terminology} We will work with schemes and ind-schemes over the complex numbers. For us, a {\bf subvariety} $V$ of a scheme $S$ will mean a locally closed subset of $S$ with the reduced scheme structure. In particular, a dense subvariety will always be an open subset of the original variety. 

\subsection{Kac-Moody Lie Algebras} We briefly review the theory of Kac-Moody Lie algebras (c.f. \cite{HongKang}). Let us fix a generalized Cartan matrix $A$, and let $I$ denote the set of nodes of the corresponding Dynkin diagram. We can form the associated  Kac-Moody Lie algebra $\g$. From $A$ we can canonically identify two dual lattices: the weight lattice $X^*\left(A\right)$ and the coweight lattice $X_*\left(A\right)$. We also have a set of simple coroots $\left\{\alpha_i\right\} \subset X_*\left(A\right)$ and simple roots $\left\{\check{\alpha}_i\right\} \subset X^*\left(A\right)$. Let $\Lambda \subset X_*\left(A\right)$ be the integral span of the simple coroots, and let $\Lambda^+ \subset \Lambda$ be the positive-integral span of the simple coroots.

The Lie algebra $\g$ is equipped with a natural triangular decomposition $\g = \n_- \oplus \ft \oplus \n_+$, where $\ft$ is the Cartan subalgebra, and $\n_+$ (resp. $\n_-$) is the direct sum of the positive (resp. negative) root spaces. As usual, $\n_+$ is generated as a Lie algebra by the Chevalley generators $\left\{ E_i\right\}_{i \in I}$. Similarly, we have Chevalley generators $\left\{F_i\right\}_{i \in I}$ for $\n_-$.  

We can also form the $\emph{formal}$ Kac-Moody Lie algebra $\overline{\g}$ by completing $\n_+$ with respect to the root grading. Explicitly, $\overline{\g} = \n_- \oplus \ft \oplus \overline{\n_+}$, where $\overline{\n_+}$ is the \emph{direct product} of the positive root spaces. From now onward, we will suppress the over-bars and write $\g$ and $\n_+$ when discussing both the minimal and formal Kac-Moody algebras. This will simplify the discussion by allowing uniform notation. 

\subsection{The Dual Lie Algebra} If $\g$ is the Kac-Moody Lie algebra corresponding to a generalized Cartan matrix $A$, then we define the {\bf dual Lie algebra} $\g^\vee$ to be the Kac-Moody Lie algebra corresponding to transpose generalized Cartan matrix $A^T$. We remark that $X^*\left(A^T\right) = X_*\left(A\right)$ and $X_*\left(A^T\right)
= X^*\left(A\right)$, i.e., the coweight and weight lattices are swapped from those of $A$. For much of this paper, we will focus on the case  $\g = \asl_n$, the untwisted affinization of $\mathfrak{sl}_n$. In this case, the generalized Cartan matrix is symmetric, and we see that $\asl_n$ is self-dual.

\subsection{Kac-Moody Groups}
We will write $G$ for the ``simply connected'' group corresponding to a minimal or formal Kac-Moody Lie algebra (c.f. \cite{Kumar}); we will refer to $G$ as either the minimal or formal Kac-Moody group. Let $N_-$, $T$ and $N_+$ be the subgroups corresponding to $\n_-$ , $\ft$ and $\n_+$ respectively. Let $B_+$ and $B_-$ be the subgroups corresponding to $\fb_- = \n_- \oplus \ft$ and $\fb_+ = \ft \oplus \n_+$ respectively. If we fix $J \subset I$, we can form the corresponding positive parabolic subalgebra $\fp^{J}_+$ by adjoining to $\fb_+$ all the negative root spaces corresponding the nodes in $J$ and taking Lie subalgebra they jointly generate. Let $P^J_+$ be the corresponding subgroup of $G$. Denote by $M^J$ the corresponding Levi factor. 

\begin{Remark}
In general, $G$ will have the structure of a group ind-scheme. When $A$ is a finite-type Cartan matrix, $G$ will be a finite dimensional complex reductive group. Furthermore, in the finite-type case, there is no difference between minimal and formal versions of Lie algebra and group.
\end{Remark}

\subsection{The Affine Grassmannian}

Let $\O = \C[[t]]$ be the ring of formal Taylor series in one variable, and let $\K = \C((t))$ be the field of formal Laurent series in one variable. Let $\C[t^{-1}]^+_k = \left\{ a_kt^{-k} + \cdots a_1t^{-1} | a_i \in \C \right\}$, which we view as a subset of $\K$. 

For any group ind-scheme $K$ (e.g. any of the groups we have defined above), we define the affine Grassmannian $\Gr_K = K(\K)/K(\O)$. In general, we can give $\Gr_K$ the structure of a set. However, when $K$ is a finite-dimensional algebraic group,  $\Gr_K$  can be naturally viewed as the $\C$-points of an ind-scheme of ind-finite type. In particular, when $K=G$ is a finite-type Kac-Moody group, we can speak of $\Gr_G$ as a geometric object. Unfortunately, for general Kac-Moody $G$ we do not currently have a good way of working with $\Gr_G$ as a geometric object. However, for the purposes of this paper we don't need the entire affine Grassmannian. Rather we only need certain subvarieties of $\Gr_G$ called Mirkovi\'c-Vilonen (MV) cycles. Fortunately, when $G$ is of untwisted affine type, we have a good geometric substitute for MV cycles coming from the theory of quasimap spaces.

\subsection{MV Cycles and Quasimap Spaces}

Let us fix a  Kac-Moody group $G$. 
To discuss MV cycles, we need to define certain subsets of $\Gr_G$.  We have a subset $T(\K)/T(\O) \subset G(\K)/G(\O)$, which is canonically identified with the coweight lattice $X_*(A) = \text{Hom}\left(T, \mathbb{G}_m\right)$. Let us denote $t^\lambda$ the point of $\Gr_G$ corresponding to the coweight $\lambda$. Consider the subgroups $N_+(\K)$ and $N_-(\K)$. We denote the {\bf positive semi-infinite cells} to be the orbits of coweight lattice under $N_+(\K)$, i.e., 
\begin{align*}
S^\lambda = N_+(\K) \cdot t^\lambda
\end{align*}
Similarly, we define the {\bf negative semi-infinite cells} to be the orbits of the coweight lattice under $N_-(\K)$.
\begin{align*}
T^\lambda = N_-(\K) \cdot t^\lambda
\end{align*}

It is easy to see that the $S^\lambda \cap S^\gamma = \emptyset $ and $T^\lambda \cap T^\gamma = \emptyset$ for $\lambda \neq \gamma$. Let us for a moment consider only finite-type $G$. In this case, these sets have the structure of ind-schemes of infinite dimension and infinite codimension in the affine Grassmannian. However, for all $\gamma, \lambda$, the intersection $S^\gamma \cap T^\lambda$ is a finite dimensional algebraic variety. Moreover the intersection $S^\gamma \cap T^\lambda = \emptyset$ unless $\gamma - \lambda \in \Lambda^+$. In this case, we can identify the irreducible components of $S^\gamma \cap T^\lambda$ with a basis in the $\lambda$-weight space of the Verma module with highest weight $\gamma$ \cite{FFKM}. We call these irreducible components {\bf MV cycles}. 

However, when $G$ is not finite type we cannot directly give these intersections the structure of an algebraic variety. However, we have a substitute in the form of quasimap spaces. Let us now consider only the case when G is a \emph{formal} Kac-Moody groups. Associated to $G$ is the Kashiwara flag scheme $\B$, which we can think of as the quotient $G/B_-$. When $G$ is finite type, this is just the flag variety of $G$. For general $G$, $\B$ still retains many of the geometric features of the finite dimensional flag variety. In particular, it has a Schubert cell decomposition. Let $U \subset \B$ denote the unique open Schubert cell, i.e. the ``big cell". Moreover, as in the finite dimensional case the second homology of $\B$ is naturally in bijection with $\Lambda$. In particular, any algebraic map $\phi : C \rightarrow \B$ from an algebraic curve into $\B$ has a well-defined degree, which we view as an element of $\Lambda^+$. Let us fix $\lambda \in \Lambda^+$, and  consider the following space (i.e. functor of points) $\oF{\lambda}$, defined in \cite{BFG}, that classifies maps $\phi : \mathbb{P}^1 \rightarrow \B$ satisfying the following conditions:

\begin{itemize}
  \item $\text{deg}\left(\phi\right) = \lambda$
  \item $\phi\left(\mathbb{P}^1 - 0\right) \subset U$
  \item $\phi\left(\infty\right) = 1_\B$ (the unit point in $\B$)
\end{itemize}

When $G$ is finite type, we have an isomorphism $\oF{\lambda} \simeq S^0 \cap T^{-\lambda}$ \cite{FM}. In particular, $\oF{\lambda}$ is a finite-dimensional scheme. When $G$ is of untwisted affine type, the authors of \cite{BFG} give a proof that $\oF{\lambda}$ is a finite-type, finite-dimensional scheme. Moreover, they prove that it is equidimensional, and they explicitly compute its dimension.

For general type, we always have the following set theoretic bijection \cite[Theorem 2.7]{BFK}:

\begin{align}
\oF{\lambda}\left(\C\right) \simeq S^0 \cap T^{-\lambda}
\end{align}

\noindent We will give details of this bijection in section \ref{BFKBijectionSection}.

Let $\cL = \bigsqcup_{\lambda \le 0} \Irr\left(\oF{\lambda}\right)$ denote the set of {\bf generalized MV cycles}. When $G$ is finite type, this corresponds exactly with our original notion of MV cycles. In the untwisted affine case, we call elements of $\Irr\left(\oF{\lambda}\right)$ {\bf double MV cycles} (because they should be isomorphic to the analogs of MV cycles in the yet-to-be-defined double affine Grassmannian).

The spaces $\oF{\lambda}$ have a natural closure $\F^\lambda$ called the quasimap closure. The closures $\F^\lambda$ are defined in \cite{BFG}. The exact definition will not matter for this paper as we will use them as an auxiliary tool to study the spaces $\oF{\lambda}$. In particular, we will use the fact that we have a canonical identification $\cL = \bigsqcup_{\lambda \le 0} \Irr\left(\F^\lambda\right)$ given by sending an irreducible component of $\oF{\lambda}$ to its closure in $\F^\lambda$.

\subsection{Review of Crystals}{\label{ReviewOfCrystals}} We will recall the definition of crystals. For later convenience, we will recall the definition of a crystal for the Lie algebra $\g^\vee$ dual to a given Kac-Moody Lie algebra $\g$ (c.f. \cite{BravGait}). A $\g^\vee${\bf -crystal} is a set $\mathbf{B}$ along with the following data:

\begin{itemize}
  \item A weight function $\wt : \mathbf{B} \rightarrow X^*(A)$
  \item For each $i \in I$, crystal operators $e_i,f_i :\mathbf{B} \rightarrow \mathbf{B} \sqcup \{0\}$
  \item For each $i \in I$, $i$-string functions $\varepsilon_i,\phi_i : \mathbf{B} \rightarrow \Z$
\end{itemize}

\noindent This data should satisfy the following axioms:

i) For all $\mathbf{b} \in \mathbf{B}$, $\phi_i(\mathbf{b}) = \varepsilon_i(\mathbf{b}) + \langle \wt\left(\mathbf{b}\right),\alpha_i\rangle$

ii) Let $\mathbf{b} \in \mathbf{B}$.  If $e_i \left(\mathbf{b}\right) \ne 0$ for some $i$, then

\begin{align*}
\wt\left(e_i\left(\mathbf{b}\right)\right)=\wt(\mathbf{b})+\check{\alpha}_i, \varepsilon_i\left(e_i\left(\mathbf{b}\right)\right) = \varepsilon_i\left(\mathbf{b}\right) - 1, \varphi_i\left(e_i\left(\mathbf{b}\right)\right) = \varphi_i\left(\mathbf{b}\right) + 1
\end{align*}

iii) Similarly, if $f_i \left(\mathbf{b}\right) \ne 0$ for some $i$, then 

\begin{align*}
\wt\left(f_i\left(\mathbf{b}\right)\right)=\wt(\mathbf{b})-\check{\alpha}_i, \varepsilon_i\left(f_i\left(\mathbf{b}\right)\right) = \varepsilon_i\left(\mathbf{b}\right) + 1, \varphi_i\left(f_i\left(\mathbf{b}\right)\right) = \varphi_i\left(\mathbf{b}\right) - 1
\end{align*}

iv) If $\mathbf{b}, \mathbf{b^\prime} \in \mathbf{B}$, then we have $\mathbf{b^\prime} = e_i \left(\mathbf{b}\right)$ if and only if $\mathbf{b} = f_i \left(\mathbf{b^\prime}\right)$.

\vspace{0.1in}

As we mentioned in the introduction, to every representation of $\g^\vee$ in the BGG category $\O$ we can canonically identify a $\g^\vee$-crystal (c.f. \cite{HongKang}). We define the $B(\infty)$ crystal to be the crystal associated to the Verma module with highest weight zero.

\section{The Braverman-Finkelberg-Gaitsgory Crystal Structure}

When $G$ is finite-type or an untwisted (formal) affine Kac-Moody group, we have the following theorem \cite{BFG} :

\begin{Theorem}
The set $\cL$ has the structure of the $B(\infty)$ crystal for the dual Lie algebra $\g^\vee$. 
\end{Theorem}

In fact, Braverman-Finkelberg-Gaitsgory define a pair of crystal structures $(\tilde{e}_i,\tilde{f}_i,\varepsilon_i,\phi_i,\wt)$ and $(\tilde{e}^*_i,\tilde{f}^*_i,\varepsilon^*_i,\phi^*_i,\wt)$ on the set $\cL$ that give it the structure of the $B(\infty)$ crystal in two different ways. In this paper, we will concern ourselves exclusively with the second crystal structure. To reduce clutter we will drop the stars and the tildes, and write $(e_i,f_i,\varepsilon_i,\phi_i,\wt)$ for what is denoted in \cite{BFG} as $(\tilde{e}^*_i,\tilde{f}^*_i,\varepsilon^*_i,\phi^*_i,\wt)$.

In this section, we will recall the definition of the BFG crystal structure and give explicit formulas for the crystal operators, directly generalizing a finite-dimensional result of Baumann and Gaussent.

When $G$ is finite type, $\Gr_G$ has a geometric structure, and we have the isomorphism of varieties $\oF{\lambda} \overset\sim\rightarrow S^0 \cap T^\lambda$. In section 13 of \cite{BFG}, the crystal structure is defined purely in the language of the affine Grassmannian. When $G$ is not finite type, we only have the quasimaps spaces, and section 14 of \cite{BFG} is a translation of the construction in the previous section to the language of quasimaps spaces. Thus sections 13 and 14 provide a dictionary between the affine Grassmannian and the quasimaps spaces for the purposes of the BFG crystal structure.

When $G$ is finite type, Baumann and Gaussent give a explicit algebraic formula for the crystal structure on MV cycles. Let us note that in Baumann-Gaussent's work, MV cycles are defined to be irreducible components of the closure $\overline{S^0 \cap T^\nu}$. Because $\Irr\left(S^0\cap T^\nu\right) = \Irr\left(\overline{S^0 \cap T^\nu}\right)$, we can easily translate their results to our match our convention. To be unambiguous, we will call irreducible components of $\overline{S^0 \cap T^\nu}$ {\bf closed MV cycles}. Baumann and Gaussent's result allow us to compute the BFG crystal operator explicitly in terms of multiplication by Chevalley subgroups.

\begin{Theorem}{\cite{baumanngaussent}}{\label{BaumannGaussent}}
Fix $i \in I$. Let $Z$ be a closed MV cycle. Let $Z^\prime = f^k_i (Z)$ be the result of applying the $f_i$ operator $k$ times to $Z$ (so that $Z^\prime$ is the corresponding \emph{closed} MV cycle). Then there exist dense locally closed subvarieties $\dot{Z} \subset Z$ and $\dot{Z}^\prime \subset Z^\prime$ such that the following holds:

The map $f:\C[t^{-1}]^+_k \times \dot{Z} \rightarrow \dot{Z}^\prime$ given by $f(p,z) = x_i\left(pt^{\phi_i(Z)}\right)z$ is well-defined and is a homeomorphism. 

Here, $x_i: \K \rightarrow N_+(\K)$ is the one-parameter Chevalley subgroup corresponding to the simple coroot $\alpha_i$, and the multiplication is the natural action of $N_+(\K)$ on the affine Grassmannian $\Gr_G$.

\end{Theorem}

\begin{Remark}
Baumann-Gaussent originally phrase their result in terms of the $B(-\infty)$-crystal. We have performed the obvious modification to rephrase their result for the $B(\infty)$-crystal. Also their theorem gives finer information than what we have given above. In particular, the subvarieties $\dot{Z} \subset Z$ and $\dot{Z}^\prime \subset Z^\prime$ have very explicit descriptions, which will become clear when we prove the generalization to the affine case. 
\end{Remark}

The goal of this section is to prove an analog of this result for double MV cycles. The proof will mostly follow the original proof of Baumann-Gaussent. However, their proof is written in the language of the affine Grassmannian, and we need to translate it to the language of quasimap spaces. Fortunately, sections 13 and 14 of \cite{BFG} provide a sufficient dictionary. 

\subsection{Recalling the BFG Construction}

We will assume familiarity with the notation and constructions of \cite{BFG}. Fix a node $i \in I$. Let $P=P^{\{i\}}$ be the corresponding positive subminimal parabolic subgroup of $G$, and let $M=M^{\{i\}}$ be the corresponding rank one Levi factor. Let $B_\pm(M)$ be the induced positive and negative Borel subgroups of $M$, and let $N_\pm(M)$ be the corresponding unipotent radicals. Let us fix $\mu \in \Lambda$. The authors construct an ind-scheme $\Stot^\mu$ with a natural action by the group $N_+(M)(\K)$ \cite[Section 14.8]{BFG}. The exact definition will not be relevant for this subsection, so we will delay discussing the details of the definition to the next subsection. 

Furthermore, the authors construct a fiber bundle $r_\mu:\Stot^\mu \rightarrow \Gr_{B(M)}^\mu$. This map $r_\mu$ is equivariant for an action of the group $N_+(M)(\K)$. Since $M$ is rank 1, we have an isomorphism $x_i:\K \overset\sim\rightarrow N_+(M)(\K)$ given by the one parameter subgroup corresponding to exponentiating the unique positive root space.

In summary we have the following diagram where the horizontal arrows are the action via $x_i$:

\begin{align}{\label{ActionDiagram}}
\begin{xy}
(0,20)*+{\K \times \Stot^\mu}="a"; (40,20)*+{\Stot^\mu}="b";%
(0,0)*+{\K \times Gr^\mu_{B(M)}}="c"; (40,0)*+{Gr^\mu_{B(M)}}="d";%
{\ar "a";"b"};
{\ar "a";"c"}?*!/_3mm/{r_\mu};
{\ar "b";"d"}?*!/_3mm/{r_\mu};%
{\ar "c";"d"};%
\end{xy}
\end{align}

Because of the group $\K$ is connected and acts transitively on the base, we can canonically identify the irreducible components of any two fibers using the group action. In fact, if $X \subset \Gr_{B(M)}^\mu$ is any irreducible subvariety, we can canonically identify the irreducible components of $r_\mu^{-1}(X)$ with the irreducible components of any fiber using the group action.

In particular, let $\lambda \in \Lambda^+$, and let $D = \Gr_{B(M)}^\mu \cap \overline{Gr_{B_-(M)}^\lambda}$. Because $M$ is rank one, $D$ is irreducible. Define $\Stot^{\mu,\leq \lambda} = r_\mu^{-1}(D)$. By the discussion in \cite{BFG}, we can identify $\Stot^{\mu,\leq \lambda}$ with a locally closed subset of $\F^\lambda$.  Moreover, these subsets are disjoint for different choices of $\mu$, and we have the following decomposition:

\begin{align*}
\F^\lambda = \bigsqcup_{\mu} \Stot^{\mu,\leq \lambda}
\end{align*}
If we denote by $\Irr^\text{top}\left(\Stot^{\mu,\leq \lambda}\right)$ the set of irreducible components of $\Stot^{\mu, \leq \lambda}$ whose dimension is the same as that of $\F^\lambda$, then we have the following decomposition: 
\begin{align*}
\Irr\left(\F^\lambda\right) = \bigsqcup_{\mu} \Irr^\text{top}\left(\Stot^{\mu,\leq \lambda}\right)
\end{align*}
And in particular, we can write 
\begin{align*}
\cL = \bigsqcup_{\lambda} \bigsqcup_{\mu} \Irr^\text{top}\left(\Stot^{\mu,\leq \lambda}\right)
\end{align*}
Using this decomposition, we can define the \cite{BFG} crystal structure as follows.

\begin{Definition}{\label{BFGCrystalDefinition}}
Let $i \in I$, and let $Z \in \cL$ be a double MV cycle. Then we define $f_i(Z)$ as follows. By the above discussion, $Z \in \Irr^{\text{top}}\left(\Stot^{\mu,\leq \lambda}\right)$ for a unique pair $\mu$ and $\lambda$. By the discussion about irreducible components of fibers of $r_\mu$, we have a canonical bijection $\Irr^{\text{top}}\left(\Stot^{\mu,\leq \lambda}\right) \cong \Irr^{\text{top}}\left(\Stot^{\mu,\leq \lambda-\alpha_i}\right)$. Define $f_i(Z)$ to be the image of $Z$ under this bijection. 

The operator $e_i$ is defined to be the standard partial inverse of $f_i$, i.e. $e_i(Z) = Z^\prime$ if there exists some (necessarily unique) $Z^\prime$ such that $Z = f_i(Z^\prime)$. Otherwise $e_i(Z) = 0$. The auxiliary data of $\epsilon_i, \phi_i, \wt$ is obviously defined.

\end{Definition}

With this definition, we can state the following theorem.

\begin{Theorem}{ \bf \cite{BFG}}
The set $\cL \cup \{0\}$ with the operators defined above form a $B(\infty)$ crystal for the dual Lie algebra $\g^\vee$.
\end{Theorem}

Moreover, we can now state and prove the following theorem, which is a natural generalization of Theorem \ref{BaumannGaussent}.

\begin{Theorem}{ \bf Generalization of the Baumann-Gaussent Formula}{\label{GeneralizationOfBaumannGaussentFormula}}
Fix $i \in I$. Let $Z$ be a irreducible component of $\F^\lambda$, and let $Z^\prime = {f^k_i} (Z)$ be the result of applying the $f_i$ operator $k$ times to $Z$. By the above discussion, we can identify $Z$ with a irreducible component of $\Stot^{\mu,\leq \lambda}$ for a unique $\mu$. Let $\dot{Z} = Z \cap \Stot^{\mu,\leq \lambda}$, let $\dot{Z}^\prime = Z^\prime \cap \Stot^{\mu, \leq \lambda - k\alpha_i}$. Then we can define the following map, which is a homeomorphism:

$f:\C[t^{-1}]^+_k \times \dot{Z} \rightarrow \dot{Z}^\prime$, given by $f(p,z) = x_i(pt^{\phi_i(Z)})\cdot z$, where $\cdot$ denotes the restriction of the action defined above.
\end{Theorem}

\begin{proof}
This proof proceeds in direct parallel to the proof of \cite[Proposition 14]{baumanngaussent}.
Let 
\begin{align*}
D = \Gr_{B(M)}^\mu \cap \overline{\Gr_{B_-(M)}^\lambda} \text{, and } D^\prime = \Gr_{B(M)}^\mu \cap \overline{\Gr_{B_-(M)}^{\lambda-k\alpha_i}}
\end{align*}
Note that because $M$ has rank one, both $D$ and $D^\prime$ are irreducible.

Using an explicit rank-one calculation in Baumann-Gaussent, \cite[Propositions 8]{baumanngaussent}, we can conclude that the map $g: \C[t^{-1}]^+_k \times D \rightarrow D^\prime$ given by $g(p,d) = x_i(pt^\phi_i(Z))\cdot z$ is a homeomorphism. Restricting diagram (\ref{ActionDiagram}) to $g$, we get the following commutative diagram:

\begin{align*}
\begin{xy}
(0,20)*+{\C[t^{-1}]^+_k  \times \Stot^{\mu, \leq \lambda}}="a"; (40,20)*+{\Stot^{\mu, \leq \lambda - k \alpha_i}}="b";%
(0,0)*+{\C[t^{-1}]^+_k \times D}="c"; (40,0)*+{D^\prime}="d";%
{\ar "a";"b"}?*!/_3mm/{f};
{\ar "a";"c"}?*!/_3mm/{r_\mu};
{\ar "b";"d"}?*!/_3mm/{r_\mu};%
{\ar "c";"d"}?*!/_3mm/{g};%
\end{xy}
\end{align*}

Notice that this square is Cartesian and that the map $r_\mu$ is faithfully flat. Because the map $g$ is a homeomorphism, and because the property of being a homeomorphism is preserved under flat base change\cite[Proposition 2.6.2]{EGA4}, we conclude the map $f$ is also a homeomorphism. But the map $f$ is precisely how we identify $\Irr^{\text{top}}\left(\Stot^{\mu,\leq \lambda}\right) \overset\sim\rightarrow\Irr^{\text{top}}\left(\Stot^{\mu,\leq \lambda-\alpha_i}\right)$ when defining the crystal structure. So we conclude that the restricted map

\begin{align*}
f:\C[t^{-1}]^+_k \times \dot{Z} \rightarrow \dot{Z}^\prime
\end{align*}
is well-defined and a homeomorphism.
\end{proof}

\subsection {\label{BFKBijectionSection}}
{\bf The bijection $\oF{\lambda}\left(\C\right) \simeq S^0 \cap T^{-\lambda}$.}

In this section we will explain in detail how the bijection $\oF{\lambda}\left(\C\right) \simeq S^0 \cap T^{-\lambda}$ works. In addition, we will prove a certain equivariance property that will be required for the remainder of the paper. This section will require manipulations of various moduli spaces from \cite{BFG} that will not appear later in the paper. We will freely use results from this paper, giving precise citations when we do. 

Recall that $\oF{\lambda}$ is defined as the moduli space of maps $\phi : \mathbb{P}^1 \rightarrow \B$ satisfying the following conditions:

\begin{itemize}
  \item $\text{deg}\left(\phi\right) = \lambda$
  \item $\phi\left(\mathbb{P}^1 - 0\right) \subset U$
  \item $\phi\left(\infty\right) = 1_\B$ (the unit point in $\B$)
\end{itemize}

Here $\B$ is the Kashiwara flag scheme, and $U$ is the open Schubert cell. By the Bruhat decompostion, we have a canonical bijection $U \simeq N$, where $N$ is the positive unipotent part of our formal Kac-Moody group $G$.

Thus we have the following sequence of morphisms:

\begin{align*}
\oF{\lambda} \rightarrow \text{ Maps}(\mathbb{P}^1 - 0 ,N) \rightarrow \text{ Maps}( \text{Spec } \K,N) 
\end{align*}

Here the first map is given by restriction of a map $\phi : \mathbb{P}^1 \rightarrow \B$ to a map $\mathbb{P}^1 -0 \rightarrow U \simeq N$, and the second map is given by restriction to the formal punctured disk around $0$ in $\mathbb{P}^1 - 0$, which we identify with $\text{Spec } \K$. Note that the we consider the second two spaces only as moduli functors and do not concern ourselves with issues of representability of these functors. 

Passing to $\C$-points, we get a map $\oF{\lambda}\left(\C\right) \rightarrow N(\K)$. Composing with the quotient $N(\K) \rightarrow N(\K)/N(\O) \simeq S^0$, we get a map $\oF{\lambda}\left(\C\right) \rightarrow S^0$. We have the following theorem, due to Braverman-Finkelberg-Kazhdan

\begin{Theorem}{\bf \cite[Theorem 2.7]{BFK}}{\label{BFKBijection}}
The image of the map $\oF{\lambda}\left(\C\right) \rightarrow S^0$ is contained in $S^0 \cap T^{-\lambda}$, and the resulting map $\oF{\lambda}\left(\C\right) \rightarrow S^0 \cap T^{-\lambda}$ is a bijection.

\end{Theorem}

For the remainder of this subsection, we will discuss how this bijection interacts with the formula from Theorem \ref{GeneralizationOfBaumannGaussentFormula}. In particular, we will be able to compute the crystal stucture on the $\C$-points of double MV cycles without having to explicitly mention the intricate geometry involved.

Recall the $N_+(M)(\K)$-equivariant fiber bundle $r_\mu:\Stot^\mu \rightarrow \Gr_{B(M)}^\mu$ used in defining the crystal structure on double MV cycles. Let $\overset{\circ} D = \Gr_{B(M)}^\mu \cap Gr_{B_-(M)}^\lambda$. Define $\Stot^{\mu,\lambda} = r_\mu^{-1}\left(\overset{\circ}D\right)$. Then we have the following decomposition into locally closed subsets (c.f the proof of \cite[Proposition 15.2]{BFG}):

\begin{align*}
\oF{\lambda} = \bigsqcup_{\mu} \Stot^{\mu,\lambda}
\end{align*}

\noindent Taking the union over all $\lambda$, we have

\begin{align*}
\bigsqcup_{\lambda} \oF{\lambda} = \bigsqcup_{\lambda} \bigsqcup_{\mu} \Stot^{\mu,\lambda} = \bigsqcup_{\mu} \Stot^\mu
\end{align*}

\noindent Passing to $\C$-points and using the Braverman-Finkelberg-Kazhdan bijection, we have a bijection $\bigsqcup_{\lambda} S^0 \cap T^{-\lambda} = \bigsqcup_{\mu} \Stot^\mu(\C)$. 

Notice that $\bigsqcup_{\lambda} S^0 \cap T^{-\lambda}$ is strictly smaller than $S^0$. In the terminology of \cite[Lemma 2.5]{BFK}, it is precisely the image of  ``good" elements of $N(\K)$ inside of of $S^0$. However, it is easy to see that $\bigsqcup_{\lambda} S^0 \cap T^{-\lambda}$ is stable under the action of $N_+(M)(\K)$ (essentially because $N_+(M)(\K)$ preserves the ``good" elements of $G(\K)$).

Thus we have a bijection $\bigsqcup_{\lambda} S^0 \cap T^{-\lambda} = \bigsqcup_{\mu} \Stot^\mu$ and actions of $N_+(M)(\K)$ on both sides of this bijection. We can now state the following proposition.

\begin{Proposition}{\label{EquivarianceProposition}} The bijection of sets $\bigsqcup_{\lambda} S^0 \cap T^{-\lambda} \rightarrow \bigsqcup_{\mu} \Stot^\mu$ is $N_+(M)(\K)$-equivariant.
\end{Proposition}

\begin{proof}
The proof requires nothing more than unwinding the various lengthy definitions made in \cite{BFG}.

First, we will need to use an alternate characterization of $\oF{\lambda}$ as a closed subscheme of a {\bf Zastava space}. Let $\AOne \subset \POne$ be the complement of $\infty$ inside $\POne$. Let $\lambda \in  \Lambda^+$. Then the Zastava space $\cZ^\lambda(\AOne)$ is the moduli of triples $(D^\theta,\mathcal{F}_N,\kappa)$ where $D^\theta$ is a ``colored divisor" of degree $\lambda$ (c.f. \cite[Section 2.3]{BFG}), $\mathcal{F}_N$ is an $N$ bundle on $\POne$, and $\kappa$ is a ``Pl\"ucker datum", which together satisfy a list of conditions detailed in \cite[Section 2.12]{BFG} (the exact conditions are not relevant for our discussion).

Of primary importance to us, is the open subscheme $\overset{\circ}{\cZ^\lambda}(\AOne) \subset \cZ^\lambda(\AOne)$, corresponding to the condition that the Pl\"ucker datum $\kappa$ consist of injective bundle maps. By \cite[Proposition 2.21]{BFG} the closed subscheme of $\overset{\circ}{\cZ^\lambda}(\AOne)$ consisting of those triples $(D^\theta,\mathcal{F}_N,\kappa)$ where $D^\theta$ is supported at $0$ is naturally isomorphic to $\oF{\lambda}$. Hence, for the remainder of this subsection, we will use this Zastava definition to work with  $\oF{\lambda}$

From the Zastava perspective, we can also reinterpret the bijection $\oF{\lambda}\left(\C\right) \simeq S^0 \cap T^{-\lambda}$. By \cite[Lemma 2.13]{BFG}, the $N$-bundle $\mathfrak{F}_N$ is canonically trivialized away from $0$. Restricting to the formal disk around $0$, we get a $N$ bundle on a formal disk trivialized on the punctured disk, which gives us a point of $N(\K)/N(\O)$. Thus we see that the action of $N_+(M)(\K)$ on $\bigsqcup_{\lambda} \oF{\lambda}(\C)$ corresponds to changing the trivialization of the $N$ bundle $\mathfrak{F}_N$ in a formal punctured disk centered at 0.

Let us now shift attention to the action of $N_+(M)(\K)$ on $\Stot^\mu$. The space  $\Stot^\mu$ is defined to be the moduli space of the following data:

\begin{itemize}

\item
A $B$-bundle $\F_B$ on $\AOne$, such that the induced $T$-bundle $\F^0_T$ is
trivialized, (in particular, the $B$ bundle has a canonical reduction to $N$, which we call $\F_N$)

\item
An $M$-bundle $\F_M$ on $\AOne$,

\item
Regular bundle maps
$\cV^{\check{\nu}}_{\F_B}\to \cU^{\check{\nu}}_{\F_M}$, and

\item
Meromorphic maps (regular away from $0$)
$\cU^{\check{\nu}}_{\F_M}\to \cL^{\check{\nu}}_{\F^0_T},$

\end{itemize}

\noindent such that

\begin{itemize}

\item 
The maps $\cU^{\check{\nu}}_{\F_M}\to \cL^{\check{\nu}}_{\F^0_T},$ form a Pl\"ucker datum for $\F_M$, and the compositions $\cV^{\check{\nu}}_{\F_B}\to \cU^{\check{\nu}}_{\F_M}\to \cL^{\check{\nu}}_{\F^0_T}$ form a Pl\"ucker datum for $\F_B$ 

\item
The compositions
$\cL^{\check{\nu}}_{\F^0_T}\to \cV^{\check{\nu}}_{\F_B}\to
\cU^{\check{\nu}}_{\F_M}\to
\cL^{\check{\nu}}_{\F^0_T}$ are the identity maps, and

\item
The induced maps
$\cL^{\check{\nu}}_{\F^0_T}\to
\cU^{\check{\nu}}_{\F_M}(\langle \mu,\check{\nu}\rangle\cdot 0)$ are regular bundle maps. 
\end{itemize}

\noindent Here $\check{\nu}$ varies through all dominant weights, and $\cV^{\check{\nu}}_{\F_B}$,$\cU^{\check{\nu}}_{\F_M}$, $\cL^{\check{\nu}}_{\F^0_T}$ are the associated bundles with fiber equal to the irreducible representation of highest weight $\check{\nu}$

In particular, we get a pair of opposite Pl\"ucker data for both $\F_B$ and $\F_M$ that are related by the maps $\cV^{\check{\nu}}_{\F_B}\to \cU^{\check{\nu}}_{\F_M}$. Moreover, because these Pl\"ucker data are ``transverse'' away from $0$, i.e. the meromorphic compositions $\cL^{\check{\nu}}_{\F^0_T}\to \cV^{\check{\nu}}_{\F_B}\to
\cU^{\check{\nu}}_{\F_M}\to
\cL^{\check{\nu}}_{\F^0_T}$ are the identity maps, both bundles $\F_B$ and $\F_M$ have canonical reductions to the maximal torus away from $0$. However, because the $T$ bundle induced from $\F_B$ is trivialized, we see that both $\F_B$ and $\F_M$ are actually trivialized away from $0$. By \cite[Proposition 14.3]{BFG}, the action of $N_+(M)(\K)$ on $\Stot^\mu$ corresponds to changing the trivialization on the bundle $\F_M$ in a formal punctured disk centered at $0$. However, because this action leaves the maps $\cV^{\check{\nu}}_{\F_B}\to \cU^{\check{\nu}}_{\F_M}$ fixed, the trivialization on the bundle $\F_B$ must change by same amount. 

Now, let us study the identification $\oF{\lambda} = \bigsqcup_{\mu} \Stot^{\mu,\lambda}$. The bundle $\F_N$ in the definition of $\oF{\lambda}$ maps to the bundle $\F_B$, which we recall has a canonical reduction to $N$. As the Pl\"ucker data are also preserved under this identification, the trivialization of  $\F_N$ away from $0$ corresponds exactly to the trivialization of $\F_B$ away from $0$. 

\end{proof}

With this equivariance established, we can prove the following corollary to Theorem \ref{GeneralizationOfBaumannGaussentFormula}.

\begin{Corollary}{\label{ExplicitCrystalFormula}}
Fix $i \in I$. Let $W$ be a irreducible component of $\oF{\lambda}$, and let $W^\prime = f^k_i(Z)$ be the result of applying the $f_i$ operator $k$ times to $W$ (so $W^\prime$ is an irreducible component of $\oF{\lambda-k\alpha_i}$. Then there exists an dense subvariety $U \subset W^\prime$, all of whose $\C$-points can be written in the form $x_i(pt^{\phi_i(Z)})\cdot z$, where $p \in \C[t^{-1}]^+_k$ and $z \in W$. Here we are viewing the $\C$-points of  $W$ (resp. $W^\prime$) as a subset of $S^0 \cap T^\lambda$ (resp. $S^0 \cap T^{\lambda - k \alpha_i}$), and the action of $x_i$ is given via the left multiplication of $N(\K)$ on $S^0$. 
\end{Corollary}

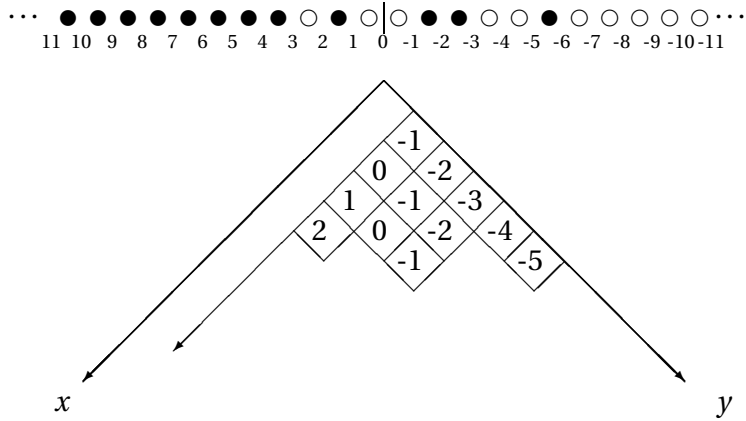
\begin{proof} Let $Z$ (resp. $Z^\prime$) be the closure of $W$ (resp. $W^\prime$) in $\F^\lambda$ (resp. $\F^{\lambda - k \alpha_i}$). We can find dense subvarieties $\dot{Z} \subset Z$ and $\dot{Z}^\prime \subset Z^\prime$ by intersecting with an appropriate $\Stot^\mu$. By Theorem \ref{GeneralizationOfBaumannGaussentFormula}, we have a homeomorphism $f:\C[t^{-1}]^+_k \times \dot{Z} \rightarrow \dot{Z}^\prime$. Let $V = \dot{Z}^\prime \cap W^\prime$. Then $U = f( f^{-1}(V) \cap \C[t^{-1}]^+_k \times \dot{Z})$ will consist of elements that can be written as $x_i(pt^{\phi_i(Z)})\cdot z$, where $p \in \C[t^{-1}]^+_k$ and $z \in W$. By Theorem \ref{BFKBijection}, we can interpret $z$ as a point in $S^0 \cap T^\lambda$, and by Proposition \ref{EquivarianceProposition} the action  $x_i(pt^{\phi_i(Z)})\cdot z$ to be that coming from left multiplication of $N(\K)$ on $S^0$.
\end{proof}

\section{Maya Diagrams and Fock Space}
For the remainder of the paper, we will let $G$ be a formal affine Kac-Moody group of type A, i.e. the group corresponding to the Lie algebra $\asl_n$ for $n \geq 2$. In this case we can identify the Dynkin diagram with the integers modulo $n$, i.e., $I = \Z/n\Z$. For these groups, we can build a very explicit representation called the Fermionic Fock space. In this section we will discuss the Maya-diagram/charged-partition bases of Fermionic Fock space, following the notation and diagrams from Tingley's expository notes \cite{Tingley}. Once we have described this basis, we will be able to write down formulas for an explicit action of $\asl_n$ on Fermionic Fock space. To begin, let us define Maya diagrams.

\begin{Definition} A {\bf left-black Maya diagram} is a sequence of white or black beads indexed by $\Z + \half$ such that all beads are black in sufficiently positive positions, and all beads are white in sufficiently negative positions. Here is an example of a left-black Maya diagram:

\setlength{\unitlength}{0.4cm}
\begin{center}
\begin{picture}(30,0.3)

\put(2.5,0){\ldots}
\put(26,0){\ldots}

\put(4.5,0){\circle*{0.5}}
\put(5.5,0){\circle*{0.5}}
\put(6.5,0){\circle*{0.5}}
\put(7.5,0){\circle*{0.5}}
\put(8.5,0){\circle*{0.5}}
\put(9.5,0){\circle*{0.5}}
\put(10.5,0){\circle*{0.5}}
\put(11.5,0){\circle*{0.5}}
\put(12.5,0){\circle{0.5}}
\put(13.5,0){\circle*{0.5}}
\put(14.5,0){\circle{0.5}}

\put(15,-0.5){\line(0,1){1}}

\put(15.5,0){\circle{0.5}}
\put(16.5,0){\circle*{0.5}}
\put(17.5,0){\circle*{0.5}}
\put(18.5,0){\circle{0.5}}
\put(19.5,0){\circle{0.5}}
\put(20.5,0){\circle*{0.5}}
\put(21.5,0){\circle{0.5}}
\put(22.5,0){\circle{0.5}}
\put(23.5,0){\circle{0.5}}
\put(24.5,0){\circle{0.5}}
\put(25.5,0){\circle{0.5}}

\put(14.8,-1){\tiny{0}}
\put(13.8,-1){\tiny{1}}
\put(12.8,-1){\tiny{2}}
\put(11.8,-1){\tiny{3}}
\put(10.8,-1){\tiny{4}}
\put(9.8,-1){\tiny{5}}
\put(8.8,-1){\tiny{6}}
\put(7.8,-1){\tiny{7}}
\put(6.8,-1){\tiny{8}}
\put(5.8,-1){\tiny{9}}
\put(4.6,-1){\tiny{10}}
\put(3.6,-1){\tiny{11}}

\put(15.6,-1){\tiny{-1}}
\put(16.6,-1){\tiny{-2}}
\put(17.6,-1){\tiny{-3}}
\put(18.6,-1){\tiny{-4}}
\put(19.6,-1){\tiny{-5}}
\put(20.6,-1){\tiny{-6}}
\put(21.6,-1){\tiny{-7}}
\put(22.6,-1){\tiny{-8}}
\put(23.6,-1){\tiny{-9}}
\put(24.4,-1){\tiny{-10}}
\put(25.4,-1){\tiny{-11}}

\end{picture}
\end{center}

\end{Definition}

\vspace{0.5cm}

Notice we have labeled the positions in increasing order from right to left following \cite{Tingley}. With this convention, all beads sufficiently left of zero must be black; hence the terminology \emph{left-black}. 

We can identify the set of all left-black Maya diagrams with downward-facing charged partitions (c.f \cite{Tingley} for the recipe). For example, we identify the previous left-black Maya diagram with the following downward-facing charged partition:

\pagebreak

\setlength{\unitlength}{0.4cm}

\begin{center}
\begin{picture}(30,0.3)

\put(2.5,0){\ldots}
\put(26,0){\ldots}

\put(4.5,0){\circle*{0.5}}
\put(5.5,0){\circle*{0.5}}
\put(6.5,0){\circle*{0.5}}
\put(7.5,0){\circle*{0.5}}
\put(8.5,0){\circle*{0.5}}
\put(9.5,0){\circle*{0.5}}
\put(10.5,0){\circle*{0.5}}
\put(11.5,0){\circle*{0.5}}
\put(12.5,0){\circle{0.5}}
\put(13.5,0){\circle*{0.5}}
\put(14.5,0){\circle{0.5}}

\put(15,-0.5){\line(0,1){1}}

\put(15.5,0){\circle{0.5}}
\put(16.5,0){\circle*{0.5}}
\put(17.5,0){\circle*{0.5}}
\put(18.5,0){\circle{0.5}}
\put(19.5,0){\circle{0.5}}
\put(20.5,0){\circle*{0.5}}
\put(21.5,0){\circle{0.5}}
\put(22.5,0){\circle{0.5}}
\put(23.5,0){\circle{0.5}}
\put(24.5,0){\circle{0.5}}
\put(25.5,0){\circle{0.5}}

\put(14.8,-1){\tiny{0}}
\put(13.8,-1){\tiny{1}}
\put(12.8,-1){\tiny{2}}
\put(11.8,-1){\tiny{3}}
\put(10.8,-1){\tiny{4}}
\put(9.8,-1){\tiny{5}}
\put(8.8,-1){\tiny{6}}
\put(7.8,-1){\tiny{7}}
\put(6.8,-1){\tiny{8}}
\put(5.8,-1){\tiny{9}}
\put(4.6,-1){\tiny{10}}
\put(3.6,-1){\tiny{11}}

\put(15.6,-1){\tiny{-1}}
\put(16.6,-1){\tiny{-2}}
\put(17.6,-1){\tiny{-3}}
\put(18.6,-1){\tiny{-4}}
\put(19.6,-1){\tiny{-5}}
\put(20.6,-1){\tiny{-6}}
\put(21.6,-1){\tiny{-7}}
\put(22.6,-1){\tiny{-8}}
\put(23.6,-1){\tiny{-9}}
\put(24.4,-1){\tiny{-10}}
\put(25.4,-1){\tiny{-11}}

\end{picture}
\end{center}

\vspace{0.2in}
\nopagebreak

\begin{center}
\begin{picture}(30,12)

\put(15,11.99){\line(-1,-1){10}}
\put(15,12){\vector(-1,-1){10}}
\put(15,12.01){\line(-1,-1){10}}
\put(4,1){$x$}

\put(15,11.99){\line(1,-1){10}}
\put(15,12){\vector(1,-1){10}}
\put(15,12.01){\line(1,-1){10}}
\put(26,1){$y$}

\put(16,11){\vector(-1,-1){8}}
\put(17,10){\line(-1,-1){4}}
\put(18,9){\line(-1,-1){3}}
\put(19,8){\line(-1,-1){3}}
\put(20,7){\line(-1,-1){1}}
\put(21,6){\line(-1,-1){1}}

\put(15,10){\line(1,-1){5}}
\put(14,9){\line(1,-1){3}}
\put(13,8){\line(1,-1){3}}
\put(12,7){\line(1,-1){1}}

\put(12.6,6.7){2}
\put(13.6,7.7){1}
\put(14.6,8.7){0}
\put(14.6,6.7){0}
\put(15.4,9.7){-1}
\put(15.4,7.7){-1}
\put(15.4,5.7){-1}
\put(16.4,8.7){-2}
\put(16.4,6.7){-2}
\put(17.4,7.7){-3}
\put(18.4,6.7){-4}
\put(19.4,5.7){-5}

\end{picture}
\end{center}

\noindent Following the recipe in \cite{Tingley}, we can label each box of the partition with an integer corresponding to which slot of the Maya diagram it lies below.

Let us define the {\bf Fermionic Fock space} $\Fock^-$ to be the formal $\C$-linear span of all left-black Maya diagrams (equivalently downward-facing charged partitions). 

\begin{Remark} Fermionic Fock space is usually defined with a basis consisting of semi-infinite wedge products. It is easy to see that the above definition is in natural bijection with the usual definition, with every Maya diagram corresponding to a particular semi-infinite wedge product (c.f. \cite{Tingley}).

\end{Remark}

Similarly, we can define right-black Maya diagrams. Again there is a bijection between right-black Maya diagrams and upward-facing charged partitions, and we denote by $\Fock^+$ its formal span. There is a natural bijection between right-black and left-black Maya diagrams coming from interchanging the roles of white and black (or equivalently flipping the charged-partition upside-down). This bijection induces a non-degenerate pairing between $\Fock^+$ and $\Fock^-$.

We will use box-numbering of charged partitions to define a representation of $\asl_n$ on $\Fock^+$ and $\Fock^-$. For the purposes of defining a representation of $\asl_n$, we will only consider the box-numberings modulo $n$. Because we will think of the affine Grassmannian as a right quotient, it will be convenient for our Lie algebras to act on the right. We define the action of $\asl_n$ on $\Fock^-$ by the following formula on Chevalley generators:

\begin{Definition}{\bf Action of $\asl_n$ on Fock Space}

Let $\gamma$ be a downward-facing charged partition. Then for all $i \in \Z / n\Z$,
\begin{align}
\vgamma E_i &:= \sum_{\small \begin{array}{c} \gamma \backslash \mu \text{ is an} \\ i \text{-colored box } \end{array}} \vmu  &  \vgamma F_i  &:= \sum_{\small \begin{array}{c} \mu \backslash \gamma \text{ is an} \\ i \text{-colored box } \end{array}} \vmu.
\end{align}

\end{Definition}

Prop 3.5.8 in \cite{Tingley} tells us that this definition gives rise to a Lie algebra action. We then  define a dual action on $\Fock^+$ by using the natural pairing with $\Fock^-$. It is easy to see that these representations are integrable, and they integrate to an action of the corresponding (minimal) Kac-Moody groups.

\begin{Remark}
A natural way to get right actions from left actions is via an anti-automorphism. There are two natural choices for $\asl_n$, the inverse and the Chevalley involution. The above formula comes from applying the Chevalley involution.
\end{Remark}

\begin{Remark} From the definitions, we see that $\Fock^+$ will be a lowest weight representation and $\Fock^-$ will be a highest weight representation. So for the formal Kac-Moody group, only an action on $\Fock^-$ will be defined. To define an action on $\Fock^+$, we need to complete it so the basis defined above becomes a topological basis.
\end{Remark}

\subsection{Valuations}

Let $V$ be a complex vector space with either the discrete topology or the topology of a vector space dual to a discrete vector space. Let us define the $\K$-vector space $V \hat{ \otimes} \K$ as follows: if $V$ is discrete, then $V \hat{ \otimes} \K = V \otimes \K $; if $V$ is dual to a discrete vector space $W$, then define  $V \hat{ \otimes} \K = \mathrm{Hom}_\C(W,\K)$.

Then we can define a function 
\begin{align*}
\val : V \hat{ \otimes} \K \rightarrow \Z \cup \{\pm \infty \}
\end{align*}
as follows. We define a decreasing filtration on $V \hat{\otimes} \K$ by subsets of the form $V \hat{\otimes} t^\ell \O$ (we define $V \hat{\otimes} t^\ell \O$ exactly as above). Let $x \in V \hat{ \otimes} \K$. If $x=0$, we set $\val(x) = \infty$. Otherwise, define $\val(x)$ to be the maximal $\ell$ such that $x \in V \hat{\otimes} t^\ell \O$; if no such $\ell$ exists, we set $\val(x) = -\infty$. Note that if $V$ is discrete, the filtration is exhaustive. In particular, $\val(x) = -\infty$ never occurs.

Taking $\K$ points of the action of $G$ on $\Fock^-$,we have an action of $G(\K)$ on $\Fock^-  \otimes \K$. Let $x_i:\K \rightarrow N_+\left(\K\right)$ be the one parameter subgroup corresponding to the simple coroot $\alpha_i$.  Let us fix $p \in \K$ with $\val(p) = \ell$, and let $x=x_i(p)$. Let $\vgamma \in \Fock^- \subset \Fock^- \otimes \K$ be a basis vector, i.e. $\gamma$ is a downward-facing charged partition. Then using the series expansion of the exponential function, we see that $\vgamma x_i(p) = \vgamma  \sum_{k=0}^{\infty} E_i^k \otimes \frac{p^k}{k!}$. Because, the action of the Chevalley generators is locally nilpotent, this sum is actually a finite sum. Using the explicit formula above, we see that $\val \left( \vgamma x_i(p) \right) = \min \left\{\val(p) \cdot B, 0 \right\}$, where $B$ is the maximal number of $i$-colored boxes that can be added to $\gamma$.

Now let $x = x_i(p)\cdot z$, where $z \in N_+(\K)$. Let us compute $\val \left( \vgamma x_i(p)\cdot z \right)$. By what we have just said, $ \vgamma x_i(p)$ consists of a finite sum of terms $ \sum \vmu \otimes a_\mu$ where $\mu$ is obtained by adding $i$-colored boxes to $\gamma$ and $\val(a_\mu) = \val(p) \cdot |\gamma \backslash \mu|$. So $\vgamma  x_i(p)\cdot z = \sum \vmu z \otimes a_\mu$. For a generic choice of $p$, none of these terms will cancel and we have proved the following lemma.

\begin{Lemma}

Consider all $p \in \K$ of a fixed valuation. Let $z \in N_+(\K)$. Then for $\vgamma \in \Fock^- \subset \Fock^- \otimes \K$ and generic choices of $p$, we have:

\begin{align}\label{FockSpaceValuationFormula}
\emph{\val} \left( \vgamma x_i(p)\cdot z \right) = \min_{\small \begin{array}{c} \mu \text{ obtained by removing } \\ i \text{-colored boxes from } \gamma \end{array}} \left\{ \emph{\val}\left(\vmu  z\right) + \emph{\val}(p) \cdot |\gamma \backslash \mu| \right\}
\end{align}

\end{Lemma}

Note, that for $v \in \Fock^-$, $\val(v \cdot -)$ makes sense as a function on $\GK / G(\O)$, because the right action of $G( \O )$ preserves the valuation. We will later see that these functions will pick out double MV cycles when we let $\vgamma$ varies through the basis of $\Fock^-$.

\section{The Naito-Sagaki-Saito Crystal}

In this section, we will recall the definition of the Naito-Sagaki-Saito (NSS) crystal defined in \cite{NSS}. Their construction is a natural extension of the theory of MV polytopes and Berenstein-Zelevinsky (BZ) data in finite-type cases (c.f \cite{Kam}). We will present their crystal structure in a way that is most streamlined for our purposes. In particular, we will rephrase their construction using the language of Maya diagrams and charged partitions. The translations from their language to our language is straightforward, but in our presentation the motivation from MV polytopes is obscured.

\begin{Definition}{\label{preNSSdef}}
Let $M_\bullet$ be collection of integers indexed by left-black Maya diagrams satisfying the following conditions. Let $\tau$ be any right-black Maya diagram, and let $I$ be an interval containing its support. We can then form a right-black Maya diagram $\gamma_I$ by inverting all the colors of $\tau$ outside of the interval $I$. We say that $M_\bullet$ is a {\bf pre-NSS datum} if there is a constant $\Theta \left(M\right)_\tau$ such that $M_{\gamma_I}=\Theta \left(M\right)_\tau$ for all $I$ sufficiently large, i.e. there exists an interval $I_0$ such that $M_{\gamma_I} = \Theta \left(M\right)_\tau$ for all intervals $I$ containing $I_0$.
\end{Definition}

The main consequence of this condition is that it allows us to define another collection of integers $\Theta \left(M\right)_\bullet$ indexed by right-black diagrams. Abusing notation, we will write $M_\tau := \Theta \left(M\right)_\tau$ for any right-black Maya diagram $\tau$.

Let us define some notation for certain special right-black Maya diagrams. For every integer $i$, let us define $\Lambda_i$ to be the right-black Maya diagram that has a black bead in every position less than $i$ and a white bead in every position greater than $i$. Let $s_i \Lambda_i$ be the right-black Maya diagram obtained by switching the colors of the beads in the  $i+\half$ and $i-\half$ positions of $\Lambda_i$. 

For example, $\Lambda_2$ is this Maya diagram:
\setlength{\unitlength}{0.4cm}
\begin{center}
\begin{picture}(30,0.3)

\put(2.5,0){\ldots}
\put(26,0){\ldots}

\put(4.5,0){\circle{0.5}}
\put(5.5,0){\circle{0.5}}
\put(6.5,0){\circle{0.5}}
\put(7.5,0){\circle{0.5}}
\put(8.5,0){\circle{0.5}}
\put(9.5,0){\circle{0.5}}
\put(10.5,0){\circle{0.5}}
\put(11.5,0){\circle{0.5}}
\put(12.5,0){\circle{0.5}}
\put(13.5,0){\circle*{0.5}}
\put(14.5,0){\circle*{0.5}}

\put(15,-0.5){\line(0,1){1}}

\put(15.5,0){\circle*{0.5}}
\put(16.5,0){\circle*{0.5}}
\put(17.5,0){\circle*{0.5}}
\put(18.5,0){\circle*{0.5}}
\put(19.5,0){\circle*{0.5}}
\put(20.5,0){\circle*{0.5}}
\put(21.5,0){\circle*{0.5}}
\put(22.5,0){\circle*{0.5}}
\put(23.5,0){\circle*{0.5}}
\put(24.5,0){\circle*{0.5}}
\put(25.5,0){\circle*{0.5}}

\put(14.8,-1){\tiny{0}}
\put(13.8,-1){\tiny{1}}
\put(12.8,-1){\tiny{2}}
\put(11.8,-1){\tiny{3}}
\put(10.8,-1){\tiny{4}}
\put(9.8,-1){\tiny{5}}
\put(8.8,-1){\tiny{6}}
\put(7.8,-1){\tiny{7}}
\put(6.8,-1){\tiny{8}}
\put(5.8,-1){\tiny{9}}
\put(4.6,-1){\tiny{10}}
\put(3.6,-1){\tiny{11}}

\put(15.6,-1){\tiny{-1}}
\put(16.6,-1){\tiny{-2}}
\put(17.6,-1){\tiny{-3}}
\put(18.6,-1){\tiny{-4}}
\put(19.6,-1){\tiny{-5}}
\put(20.6,-1){\tiny{-6}}
\put(21.6,-1){\tiny{-7}}
\put(22.6,-1){\tiny{-8}}
\put(23.6,-1){\tiny{-9}}
\put(24.4,-1){\tiny{-10}}
\put(25.4,-1){\tiny{-11}}

\end{picture}
\end{center}

\vspace{0.2in}

\noindent And $s_2 \Lambda_2$ is this one:

\setlength{\unitlength}{0.4cm}
\begin{center}
\begin{picture}(30,0.3)

\put(2.5,0){\ldots}
\put(26,0){\ldots}

\put(4.5,0){\circle{0.5}}
\put(5.5,0){\circle{0.5}}
\put(6.5,0){\circle{0.5}}
\put(7.5,0){\circle{0.5}}
\put(8.5,0){\circle{0.5}}
\put(9.5,0){\circle{0.5}}
\put(10.5,0){\circle{0.5}}
\put(11.5,0){\circle{0.5}}
\put(12.5,0){\circle*{0.5}}
\put(13.5,0){\circle{0.5}}
\put(14.5,0){\circle*{0.5}}

\put(15,-0.5){\line(0,1){1}}

\put(15.5,0){\circle*{0.5}}
\put(16.5,0){\circle*{0.5}}
\put(17.5,0){\circle*{0.5}}
\put(18.5,0){\circle*{0.5}}
\put(19.5,0){\circle*{0.5}}
\put(20.5,0){\circle*{0.5}}
\put(21.5,0){\circle*{0.5}}
\put(22.5,0){\circle*{0.5}}
\put(23.5,0){\circle*{0.5}}
\put(24.5,0){\circle*{0.5}}
\put(25.5,0){\circle*{0.5}}

\put(14.8,-1){\tiny{0}}
\put(13.8,-1){\tiny{1}}
\put(12.8,-1){\tiny{2}}
\put(11.8,-1){\tiny{3}}
\put(10.8,-1){\tiny{4}}
\put(9.8,-1){\tiny{5}}
\put(8.8,-1){\tiny{6}}
\put(7.8,-1){\tiny{7}}
\put(6.8,-1){\tiny{8}}
\put(5.8,-1){\tiny{9}}
\put(4.6,-1){\tiny{10}}
\put(3.6,-1){\tiny{11}}

\put(15.6,-1){\tiny{-1}}
\put(16.6,-1){\tiny{-2}}
\put(17.6,-1){\tiny{-3}}
\put(18.6,-1){\tiny{-4}}
\put(19.6,-1){\tiny{-5}}
\put(20.6,-1){\tiny{-6}}
\put(21.6,-1){\tiny{-7}}
\put(22.6,-1){\tiny{-8}}
\put(23.6,-1){\tiny{-9}}
\put(24.4,-1){\tiny{-10}}
\put(25.4,-1){\tiny{-11}}

\end{picture}
\end{center}

\vspace{0.2in}

\begin{Remark}
The above notation corresponds to the fact that we can identify $\left\{\Lambda_i\right\}_{i\in \Z}$ with the fundamental weights for $\widetilde{GL_\infty}$, and $\left\{s_i\right\}$ with the simple reflections in its Weyl group. 
\end{Remark}

With this notation in hand, we can define the action of the ``$A_\infty$ Kashiwara operators'' on pre-NSS data:

\begin{Definition}
Let $M_\bullet$ be a pre-NSS datum, and let $i$ be an integer. We define a new pre-NSS datum $f_i(M)$ as follows. Let $c_i(M) = M_{\Lambda_i} - M_{s_i\Lambda_i} - 1$. Then we define:

\begin{align}
\left(\tilde{f}_iM\right)_\gamma = \min_{\small \begin{array}{c} \mu \text{ obtained by removing } \\ i \text{-colored boxes from } \gamma \end{array}} \left\{ M_\mu + |\gamma \backslash \mu| \cdot c_i(M) \right\}
\end{align}

Notice that there will be at most one removable box of color $i$, so the min in the definition is taken over a set of either one or two elements.
\end{Definition}

\begin{Remark}
The definition above differs slightly in form from the one in \cite{NSS}. However, an immediate calculation shows that they are equivalent.
By the second part of \cite[Proposition 3.3.2]{NSS}, we see that the definition makes sense, i.e. $f_iM$ satisfies the pre-NSS condition.
\end{Remark}

Let us fix an integer $n\ge2$. We have a natural operator $\sigma$ on the set of left-black Maya diagrams given by shifting the positions by $n$. We say a pre-NSS datum $M$ is {\bf n-periodic}, if $M_{\sigma(\gamma)} = M_\gamma$ for all left-black Maya diagrams $\gamma$. Restricting ourselves to the set of n-periodic pre-NSS data, we define the set of $\asl_n$ Kashiwara operators as follows.

\begin{Definition}
Let $i \in \Z/n\Z$, and let $M$ be an n-periodic pre-NSS datum. We define

\begin{align}
\hat{f}_{i} M = \left(\prod_{k \in \Z} \tilde{f}_{i+kn}\right) \left( M \right)
\end{align}

The right hand side of the definition requires some explanation. First we note that the operators $\tilde{f}_{i+kn}$ commute. Furthermore, for every right-black Maya diagram $\gamma$, there exists a finite-length interval $I_0 \subset \Z + \half$ such that $\left(\prod_{k \in I} \tilde{f}_{i+kn} \right) \left( M \right)_\gamma = \left(\prod_{k \in I_0} \tilde{f}_{i+kn} \right) \left( M \right)_\gamma$ for all intervals $I$ containing $I_0$. We therefore define $\left(\prod_{k \in \Z} \tilde{f}_{i+kn} \right) \left( M \right)_\gamma = \left(\prod_{k \in I_0} \tilde{f}_{i+kn} \right) \left( M \right)_\gamma$

\end{Definition}

Along with these operators, we have the additional data required to define a crystal. The $\hat{e}_i$ operators are defined in the obvious way. To every $n$-periodic pre-NSS datum $M$, we associate a coweight $\wt(M) := \sum_{i \in \Z/n\Z} M_{\Lambda_i} \cdot \hat{h}_{i}$ where $\left\{ \hat{h}_{i} \right\}$ are the simple coroots for $\asl_n$. We define $\hat{\varepsilon}_{i}(M) := -M_{\Lambda_i} - M_{s_i\Lambda_i} + M_{\Lambda_{i-1}} +M_{\Lambda{i+1}}$. And we define $\hat{\phi}_{i}(M) = \left\langle \wt(M), h_{i} \right\rangle + \hat{\varepsilon}_{i}(M)$. With these definitions in place, we can give an explicit formula for the $\hat{f}_{i}$ operators.

\begin{Lemma}{\label{FormulaForNSSOperators}}
Let $M_\bullet$ be a $n$-periodic pre-NSS datum. Then
\begin{align}
\left(\hat{f}_iM\right)_\gamma = \min_{\small \begin{array}{c} \mu \text{ obtained by removing } \\ i \text{-colored boxes from } \gamma \end{array}} \left\{ M_\mu + |\gamma \backslash \mu| \cdot \left(\hat{\phi}_{i}(M)-1\right) \right\}
\end{align}
\end{Lemma}

\begin{proof}
A straightforward calculation shows $c_{i+kn}( \tilde{f}_{i+nl} N) = c_{i+kn} (N)$ for $k \ne l$ and any n-periodic pre-NSS datum $N$. So all the $c_{i+nk}(N)$ that appear when applying the various $\tilde{f}_{i+nk}$ operators are equal to $c_i(M)$. Each $\tilde{f}_{i+nk}$ operator acts by taking minimum over the two possibilities of either adding a box of color $i+nk$ or not adding a box. When we take the infinite product, we get a minimum over the possibilities of removing any number of boxes of color congruent to $i$ modulo $n$. Because the numbers $c_{i+kn}$ stay constant at each step, we get the following formula:
\[ (\hat{f}_i M)_\gamma = \min\{ M(Z)_\mu + |\gamma / \mu| \cdot  c_i(M) \} \]
The final step is to unwind the definition of $\hat{\phi}_{i}(M)$ to see that $c_i(M) = \hat{\phi}_{i}(M)-1$.

\end{proof}

We end this section by stating the main result of \cite{NSS}. Let $O_\bullet$ be the pre-NSS data that assigns the value zero to each left-black Maya diagram. This clearly satisfies the pre-NSS condition and is $n$-periodic for any $n$. For any $n \geq 2$, let us define the set of {\bf NSS data} $\cN = \cN_n$ to be the set of all pre-NSS data obtained by applying a sequence of $\hat{f}_i$ operators to $O$.

\begin{Theorem}{\cite{NSS}}{\label{NSSCrystalStructure}}
Let $n \ge 3$. Then the set $\cN \cup \{0\}$, along with the data of $\hat{f}_{i}, \hat{e}_i, \wt, \hat{\varepsilon}_{i}, \hat{\phi}_{i}  $ form a crystal that is isomorphic to the $B( \infty )$ crystal for $\asl_n$.
\end{Theorem}

Their proof is purely combinatorial, depending on a result of Stembridge that only applies for $n \ge 3$. In the next section we will give an independent geometric proof of this result that includes the case $n=2$.

\section{Geometric Realization of NSS Data}

In this section we will show how to construct the set of NSS data from double MV cycles for $\widehat{SL}_n$. Moreover, we will give an independent proof that the NSS crystal is the $B\left(\infty\right)$-crystal. In particular, our proof will work in the case of $\asl_2$, where the combinatorial methods of \cite{NSS} do not apply.

Let us fix an integer $n \ge 2$, and let $G = \widehat{SL}_n$. Fix a coweight $\lambda \in \Lambda^+$. For each left-black Maya diagram $\gamma$, we define a function $D_\gamma : \oF{\lambda}(\C) \rightarrow \Z$ as follows: $D_\gamma([g]) = \val(\vgamma \cdot g)$. Here, $[g]$ is an element of $S^0 \cap T^{-\lambda}$, which we view as a subset of $\GK/G(\O)$.

\begin{Proposition}{\label{ConstructibilityProposition}}
The functions $D_\gamma$ are constructible.
\end{Proposition}

\begin{proof}
By the discussion of Zastava spaces in section \ref{BFKBijectionSection}, $\oF{\lambda}$ maps into the moduli space of $N$ bundles on $\POne$ trivialized away from $0$ (we view this space only as a functor, ignoring issues of representablity). In particular, we have a universal $N$ bundle $\F_N$ on $\POne \times \oF{\lambda}$ that is trivialized on $(\POne - 0) \times \oF{\lambda}$. Let us form the associated vector bundle with fibers equal to the Fock space $\Fock^-$. As $N$ is pro-unipotent, $\F_N$ trivializes on any affine open cover, and there is no difficulty in constructing this associated bundle. Then $\vgamma \in \Fock^-$ determines a section of this bundle on $(\POne - 0) \times \oF{\lambda}$ by means of the trivialization. 

Then the function $D_\gamma$ on $\oF{\lambda}$ is given by taking the order of the pole of this section on $(\POne - 0) \times x$ for any point $x \in \oF{\lambda}$. We can check this explicitly by restricting to the formal punctured disk around $0$ in $\POne$  and choosing a trivialization of $\F_N$ that extends over the full (non-punctured) formal disk. The difference of the two trivializations on the punctured disk gives us a lift of the point $x$ to $N(\K)$ via the Zastava interpretation of the bijection $\oF{\lambda}\left(\C\right) \simeq S^0 \cap T^{-\lambda}$ (c.f the proof of Proposition \ref{EquivarianceProposition}). Using this we can easily see that the order of the pole of the section of the associated bundle coming from $\vgamma$ is exactly $D_\gamma(x)$.

Now, if we had a global trivialization of $\F_N$, then this section could be viewed as a map from $\F^\lambda$ to $\Fock^- \hat{\otimes} \K$, which we view as an ind-scheme (here we are identifying $\text{Spec }\K$ with the formal punctured disk centered at $0$ in $\POne$). The function $\val : \Fock^- \hat{\otimes} \K \rightarrow \Z \cup {+\infty}$ is constructible (in fact, it has locally closed fibers), and the function $D_\gamma$ would be the pull back of this function to $\F^\lambda$ via this section.

We do not have a global trivialization, but this bundle will trivialize on on any affine cover. So we have proved that $D_\gamma$ is constructible when restricted to every affine open, from we easily see that $D_\gamma$ is in fact constructible on the whole of $\oF{\lambda}$.
\end{proof}

Because the function is constructible, on each double MV cycle $Z$, there will be a dense open subset on which it takes a single value, which we call the {\bf generic value} of $D_\gamma$. To $Z$ we can associate a pre-NSS datum $M(Z)$ defined as follows. For every left-black Maya diagram $\gamma$, we set:
\begin{align}
M(Z)_\gamma = \text{ the generic value of } D_\gamma \text{ on } Z
\end{align}

Abusing notation, we will write $M(U)_\gamma = M(Z)_\gamma$ if $U$ is a dense subvariety of Z.

\begin{Remark}
The idea of defining these constructible functions in the case of a finite-dimensional group is due to Kamnitzer \cite{Kam}, who credits it to a discussion with D. Speyer.
\end{Remark}

\begin{Proposition}{\bf Geometric Construction of NSS Crystal Operator}{\label{GeometricConstructionOfNSSOperator}}
Let $Z$ be a double MV cycle, and let $M\left(Z\right)$ be its associated pre-NSS datum. Then,
\begin{align}
\hat{f}_{i} M\left(Z\right)_\bullet = M\left(f_{i} Z\right)_\bullet
\end{align}
Here $f_i$ is the geometrically defined crystal operator on double MV cycles, and $\hat{f}_{i}$ is the combinatorially defined crystal operator from the previous section.
\end{Proposition}

\begin{proof} 
By Corollary \ref{ExplicitCrystalFormula}, $f_iZ$ contains a dense subvariety $U$ whose $\C$-points can each be written in the form $x_i(p) \cdot z$, where $\val(p) = \phi_i(Z) - 1$ and $z \in Z$. Because $U$ is dense, we have  

\begin{align*}
M\left(f_i Z\right)_\gamma = M\left(U\right)_\gamma
\end{align*}
Using the fact that every point of $U$ is of the form mentioned above, we can use Lemma \ref{FockSpaceValuationFormula} to compute:

\begin{align*}
 M\left(U\right)_\gamma = 
\min_{\small \begin{array}{c} \mu \text{ obtained by removing } \\ i \text{-colored boxes from } \gamma \end{array}} \left\{ M_\mu + |\gamma \backslash \mu| \cdot \left(\hat{\phi}_{i}\left(M\right)-1\right) \right\},
\end{align*}
which is precisely equal to $\hat{f_i} M\left(Z\right)_\bullet$ by Lemma \ref{FormulaForNSSOperators}.

\end{proof}

This geometric construction also explains the somewhat mysterious operator $\Theta$ from Definition \ref{preNSSdef}. Let $M = M\left(Z\right)$ be the pre-NSS datum corresponding to a double MV cycle $Z$. Then for any right-black Maya diagram $\tau$, we can define $\Theta\left(M\right)_\tau$ as before. For right-black diagrams $\tau$, we can define $D_\tau$ analagously as we did for left-black diagram using the action on $\Fock^+$. As before, $D_\tau$ is constructible:

\begin{Proposition}
The functions $D_\tau$ are constructible.
\end{Proposition}

\begin{proof}
The proof of Proposition \ref{ConstructibilityProposition} carries through except for one subtlety. The valuation function $\val: \Fock^+ \hat{\otimes}\K \rightarrow \Z \cup {\pm \infty}$ is no longer constructible. However, the fibers of each point except $-\infty$ are locally closed, and by the discussion of ``good'' elements in \cite{BFK}, the function $D_\gamma$ never takes the value $-\infty$ on $\oF{\lambda}$. So the argument still works.
\end{proof}

Then a calculation using charged partitions gives us the following very nice formula: 

\begin{Proposition}
$\Theta\left(M\right)_\tau = \text{ the generic value of } D_\tau \text{ on } Z$
\end{Proposition}

\begin{proof}
Let $Z$ be a double MV cycle as above, and let $M = M \left(Z\right)$ be the associated NSS datum. Let $\tau$ be a right-black Maya diagram. Then recall that $\Theta\left(M\right)_\tau = M_{\gamma_I}$, where $I$ is a sufficiently large interval in $\Z$, and $\gamma_I$ is the left-black Maya diagram obtained by inverting all the colors of $\tau$ outside the interval $I$.

Let us now interpret this in terms of charged partitions. We view diagram $\tau$ as an upward-facing charged partition, and $\gamma_I$ is a downward-facing charged partition. Both of these partitions are colored by the integers modulo $n+1$. Then it is easy to see that removing any number of boxes of a fixed color from $\gamma_I$ corresponds bijectively to adding boxes of the same color to $\tau$. Moreover, there is an integer $N$ such that for any sequence $(i_1, \cdots i_N)$ of colors, the process of removing boxes of color $i_1$, followed by removing boxes of color $i_2$, and so on up to removing boxes of color $i_N$ from $\gamma_I$ corresponds bijectively to a process of adding boxes of the same colors to $\tau$. Furthermore, by choosing $I$ large enough, we can choose $N$ arbitrarily large. 

However, removing boxes from $\gamma_I$ is exactly how the Chevalley generators act on $\Fock^-$, and adding boxes to $\tau$ is how the Chevalley generators act on $\Fock^+$. We thus see that $D_{\gamma_I}\left(z\right) = D_\tau\left(z\right)$ agree for all $z \in Z$ that can be written using a product of fewer than $N$ Chevalley subgroups. But this is always true for $z$ in an open dense subset of $Z$ by Corollary \ref{ExplicitCrystalFormula} and the fact that every double MV cycle is the result of applying a finite number of crystal operators to the unique double MV cycle of weight $0$.

\end{proof}

\vspace{0.3in}

As before, let $\cL$ denote the set of double MV cycles, and let $\cN$ denote the set of NSS data. The previous lemma implies that the operation $M$ defines a map $M : \cL \rightarrow \cN$, i.e. the pre-NSS datum associated to an double MV cycle is in fact an NSS datum. Moreover, by the definition of NSS data, we see that this map is surjective. With these observations, we state our main theorem.

\begin{Theorem} The map $M$ is an isomorphism of crystals. 

\end{Theorem}

\begin{proof}
In light of the previous proposition, we only need to prove that the map $M$ is a bijection. As we observed above, $M$ is surjective, so we only need to prove injectivity. 

First notice by induction that the $i$-string functions $\varepsilon_i$ and $\phi_i$  and the weight function $\wt$ also commute with the map $M$.

We proceed by induction on the weight (more precisely on the height of the weight). Because there is only one double MV cycle of weight 0 corresponding to the unit point $1 \in N(\K)$, we immediately see that this double MV cycle is determined by its NSS datum. Thus we have the base case for our induction. 

Let $W$ and $Z$ be double MV cycles of strictly negative weight with $M(W) = M(Z)$, and suppose $M$ is a injection for all double MV cycles of larger weight. Because $W$ and $Z$ are not weight $0$, by the general structure of the $B(\infty)$ crystal, there exists an i and j so that $e_i(Z) \neq 0$ and $e_j(W) \neq 0$ (recall the operators $e_i$ and $e_j$ are the partial inverses of $f_i$ and $f_j$ from Definition \ref{BFGCrystalDefinition}). In the $B(\infty)$ crystal, $e_i(Z) \neq 0$ iff $\varepsilon_i(Z) \neq 0$. Moreover, $\varepsilon_i(Z) = \varepsilon_i(W)$ because this quantity can be computed combinatorially from the NSS datum. 

So there exists $i$ such that $e_i(Z) \neq 0$ and $e_i(W) \neq 0$. Thus we can write $W = f_ie_i(W)$ and $Z=f_ie_i(Z)$. Now we can apply Propostion \ref{GeometricConstructionOfNSSOperator} to get $M(W) = \hat{f}_{i}M(e_i(W))$ and $M(Z) = \hat{f}_{i}M(e_i(Z))$. Since $M(W) = M(Z)$, we can apply $\hat{e}_i$ (the partial inverse to $\hat{f}_i$ from Theorem \ref{NSSCrystalStructure}) to get $M(e_i(W)) = M(e_i(Z))$. By induction, $e_i(W) = e_i(Z)$. Applying $f_i$ to both sides, we get $W=Z$
\end{proof}

From this theorem, we immediately deduce the following corollary.

\begin{Corollary}
The NSS crystal $\cN$ is the $B\left(\infty\right)$ crystal for $\asl_n$. In particular, this holds for the previously unknown case of $\asl_2$.
\end{Corollary}

%--------------------------------------------------------------------
\end{document}